\documentclass{amsart}
\usepackage{graphicx} 
\usepackage{amssymb}
\usepackage{amsmath}
\allowdisplaybreaks
\usepackage{amsthm}
\usepackage{fourier}
\DeclareMathAlphabet{\mathcal}{OMS}{cmsy}{m}{n}
\usepackage{amsfonts}
\usepackage[backend=bibtex,
    sorting=anyt,
    isbn=false,
    url=false,
    doi=false,
    maxbibnames=100]{biblatex}
\bibliography{references} 

\usepackage{tikz-cd}
\usepackage{tikz}
\usepackage{hyperref}

\newtheorem{thm}{Theorem}[section]

\theoremstyle{remark}

\begin{document}
\title{Binomial coefficients and the operadic butterfly}

\begin{abstract}
We prove the conjecture of Laubie on dimensions of components of two operads completing Loday's operadic butterfly diagram. For one of these operads, the dimension of the $n$-th component is equal to $\binom{2n}{n-1}$ for all $n\ne 3$, and for the other it is equal to $\binom{2n}{n-1}+1$ for all $n\ge 5$.    
\end{abstract}

\author{Vladimir Dotsenko}

\address{Institute of Mathematics, Academy of Mathematics and Systems Science, Chinese Academy of Sciences, Beijing, 100190, China}
\address{Institute of Mathematics and Mathematical Modeling, Almaty, Kazakhstan}

\email{vdotsenko@unistra.fr}

\author{Medet Jumadildayev}

\address{Institute of Mathematics and Mathematical Modeling, Almaty, Kazakhstan}

\address{Nazarbayev University, Astana, Kazakhstan}

\email{medetdzhuma@gmail.com}

\maketitle

\section{Introduction}

In \cite{Loday} Loday considered the following diagram expressing relationships between different types of algebras that emerged in his work at various points:
 
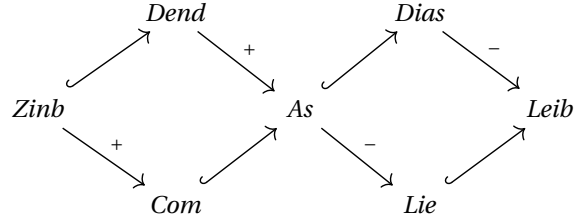
\[
\begin{tikzcd}[row sep=2.5em, column sep=2.5em]
        & \mathit{Dend} 
        \arrow[dr, "+"] & & \mathit{Dias} \arrow[dr, "-"] & \\
        \mathit{Zinb} \arrow[ur, hook] \arrow[dr, "+"] & & \mathit{As} \arrow[ur, hook] \arrow[dr, "-"] & & \mathit{Leib} \\
        & \mathit{Com} \arrow[ur, hook] 
        & & \mathit{Lie} \arrow[ur, hook] & 
\end{tikzcd}
 \]
Here $\mathit{Com}$ denotes commutative associative algebras, $\mathit{Lie}$ denotes Lie algebras, $\mathit{Leib}$ denotes Leibniz algebras (first introduced by Bloh \cite{MR193114} under the name ``$D$-algebras'' and then independently rediscovered by Loday; it appeared in print in an article of his student Cuvier \cite{MR1133486}), $\mathit{As}$ denotes associative algebras, $\mathit{Zinb}$ denotes Zinbiel algebras (originally introduced under the name ``dual Leibniz algebras'' by Loday \cite{MR1379265} and then renamed Zinbiel following a suggestion of Lemaire), $\mathit{Dend}$ denotes dendriform algebras (originally defined as dual dialgebras by Loday \cite{MR1345436}), and $\mathit{Dias}$ denotes diassociative algebras (originally introduced as dialgebras by Loday \cite{MR1345436}). 

All arrows in this diagram are functors between categories of algebras. Zinbiel algebras are a particular class of dendriform algebras, and similarly associative algebras are a particular class of diassociative algebras, commutative associative algebras are a particular class of associative algebras, and Lie algebras are a particular class of Leibniz algebras. The plus sign on the arrows corresponds to adding operations: the sum of operations in a dendriform algebra is associative, and the anticommutator in a Zinbiel algebra is associative and commutative. Finally, the minus sign corresponds to subtracting operations: a suitable difference of operations in a diassociative algebra satisfies the Leibniz identity, and the commutator in an associative algebra defines a Lie algebra structure. 

Nowadays it is common to encode types of algebras using operads \cite{LodayVallette}; the above diagram arises from maps between the corresponding operads, if one reverses all the arrows: 
 
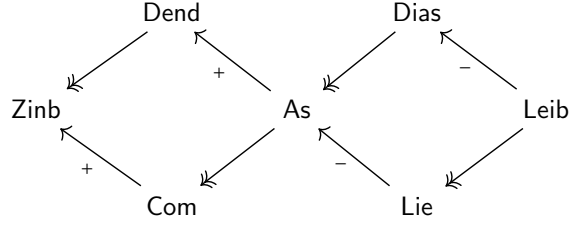
\[
\begin{tikzcd}[row sep=2.5em, column sep=2.5em]
        & \mathsf{Dend} 
        \arrow[dl, two heads] & & \mathsf{Dias}  \arrow[dl, two heads] & \\
        \mathsf{Zinb}  & & \mathsf{As}\arrow[ul, "+"] \arrow[dl, two heads]   & & \mathsf{Leib} \arrow[ul, "-"]\arrow[dl, two heads]\\
        & \mathsf{Com} \arrow[ul, "+"] 
        & & \mathsf{Lie}\arrow[ul, "-"]  & 
\end{tikzcd}
 \]
For instance, the functor $\mathit{As}\overset{-}{\longrightarrow}\mathit{Lie}$ corresponds to the map of operads $\mathsf{Lie}\to\mathsf{As}$ sending the binary operation $[a_1,a_2]$ in $\mathsf{Lie}$ to the operation $a_1a_2-a_2a_1$ in $\mathsf{As}$. Because of the shape of the diagram, Loday referred to it as the ``operadic butterfly''.

This diagram possesses an order-two symmetry given by the Koszul duality for operads \cite{MR1301191}: we have
\begin{gather*}
\mathsf{Dend}^!\cong\mathsf{Dias},\quad 
\mathsf{Zinb}^!\cong\mathsf{Leib},\\
\mathsf{As}^!\cong\mathsf{As},\quad 
\mathsf{Com}^!\cong\mathsf{Lie}
\end{gather*}
Moreover, all operads in this diagram are Koszul. 

The question raised in \cite{Loday} was whether there exists a Koszul operad $\mathcal X$ such that $\mathcal X^! \cong \mathcal X$, such that dendriform algebras are a subclass of $\mathcal X$-algebras, and diassociative algebras are obtained from $\mathcal X$-algebras by forming certain sums of generating operations. In \cite[Th.~4.1]{Loday}, it is established that there are two candidate solutions $\mathcal X^{\pm}$ which correspond to algebras with four structure binary operations $x$, $y$, $z$, and $t$ and relations
\begin{gather*}
    x(x(a_1, a_2), a_3) = x(a_1, x(a_2, a_3)) + x(a_1, t(a_2, a_3)),\\
    x(x(a_1, a_2), a_3) = x(a_1, y(a_2, a_3)) + x(a_1, z(a_2, a_3)),\\
    x(z(a_1, a_2), a_3) = z(a_1, x(a_2, a_3)) + z(a_1, t(a_2, a_3)),\\
    z(x(a_1, a_2), a_3) = z(a_1, z(a_2, a_3)) + z(a_1, y(a_2, a_3)),\\
    z(z(a_1, a_2), a_3) = z(a_1, z(a_2, a_3)) + z(a_1, y(a_2, a_3)),\\
    x(t(a_1, a_2), a_3) = t(a_1, x(a_2, a_3)),\\
    x(t(a_1, a_2), a_3) = t(a_1, z(a_2, a_3)),\\
    x(y(a_1, a_2), a_3) = y(a_1, x(a_2, a_3)),\\
    z(t(a_1, a_2), a_3) = y(a_1, z(a_2, a_3)),\\
    z(y(a_1, a_2), a_3) = y(a_1, z(a_2, a_3)),\\
    t(x(a_1, a_2), a_3) + t(t(a_1, a_2), a_3) = t(a_1, t(a_2, a_3)),\\
    t(x(a_1, a_2), a_3) + t(t(a_1, a_2), a_3) = t(a_1, y(a_2, a_3)),\\
    t(z(a_1, a_2), a_3) + t(y(a_1, a_2), a_3) = y(a_1, t(a_2, a_3)),\\
    y(x(a_1, a_2), a_3) + y(t(a_1, a_2), a_3) = y(a_1, y(a_2, a_3)),\\
    y(z(a_1, a_2), a_3) + y(y(a_1, a_2), a_3) = y(a_1, y(a_2, a_3)),\\
    y(z(a_1, a_2), a_3) - y(x(a_1, a_2), a_3) = \pm \bigl(x(a_1, t(a_2, a_3)) - x(a_1, y(a_2, a_3))\bigr).\\
\end{gather*}

These operads are regular: the operations have no symmetry and all identities have the arguments in the same order. For that reason, for each $n\ge 1$, we have 
 \[
\mathcal{X}^{\pm}(n)\cong \mathcal{X}^{\pm}_n\otimes\mathbb{Q} S_n ,
 \]
where $\mathcal{X}^{\pm}_n$ is a vector space of generators of the free $\mathbb{Q} S_n$-module $\mathcal{X}^{\pm}(n)$; in the language of  operad theory, $\mathcal{X}^{\pm}_n$ is the $n$-th component of the corresponding \emph{nonsymmetric operad}. Working with nonsymmetric operads is much easier for computations, since one avoids the blow-up of dimensions coming from actions of symmetric groups.  

In \cite[Sec.~4.2]{Loday}, the question of whether the operads $\mathcal{X}^{\pm}$ have the Koszul property was raised; it is  indicated there that the Koszul property implies that $\dim \mathcal{X}^{\pm}_n=4^{n-1}$. Recently, Laubie \cite[Prop.~2.4.3.3]{Laubie} showed that $\dim\mathcal X^+_4=58$ and $\dim\mathcal X^-_4=56$, which is different from the expected dimension $4^{4-1}=64$. (The premise that the Koszul property allows one to recover all dimensions is perhaps somewhat unfounded; we address it below in Section~\ref{sec:Koszul}.) 

In fact, in the proof of \cite[Prop.~2.4.3.3]{Laubie}, $\dim_{\mathbb{Q}} \mathcal{X}^{\pm}_n$ was computed for $n \leq 7$ using the \texttt{Haskell} operadic Gröbner bases calculator \cite{Dotsenko-Heijltjes}: 
 \[
\begin{array}{|c|c|c|c|c|c|c|c|c|}
\hline
n     & 1&2&3&4&5&6&7 \\
\hline
\dim_{\mathbb{Q}}\mathcal X^-_n   &1&4&16&56&210&792&3003  \\
\hline
\dim_{\mathbb{Q}}\mathcal X^+_n   &1&4&16&58&211&793&3004 \\
\hline 
\end{array}
 \]
These data led to a question \cite[Conj.~2.4.3.4]{Laubie} whether for all $n\ne 3$ we have 
 \[
\dim_{\mathbb{Q}}\mathcal X^-_n=\binom{2n}{n-1}. 
 \]
We answer this question in the affirmative, and also establish that for $n\ge 5$ we have
 \[
\dim_{\mathbb{Q}}\mathcal X^+_n=\binom{2n}{n-1}+1. 
 \]
It is clear that the dimensions of components of our operads only depend on the characteristic of the ground field and not the field itself; in Section \ref{sec:char-p}, we also describe how the results need to be amended if we consider our operads over a field $\mathbb{F}_p$. 

These results are proved using Gröbner bases for operads \cite{Dotsenko-Khoroshkin}. We managed to find an order of monomials for which the operad $\mathcal X^-$ has a finite Gröbner basis; it is proved by Khoroshkin and Piontkovski \cite{MR3301915}, extending previous purely combinatorial work of Rowland \cite{MR2645188},  that the Hilbert series
 \[
f_{\mathcal{P}}(q):=\sum_{n \geq 1} \dim\mathcal{P}_n q^n  
 \]
of a nonsymmetric operad $\mathcal{P}$ with a finite Gröbner basis is algebraic, and we use a version of their approach to find a closed formula for the Hilbert series $f_{\mathcal{X}^-}(q)$.
For the same ordering of monomials, the operad~$\mathcal{X}^+$ turns out to have an infinite but manageable Gröbner basis, which we could use to prove the dimension formula. 

\subsection*{Acknowledgments. } We thank Paul Laubie for useful discussions and for sharing the input file for \cite{Dotsenko-Heijltjes} which he used in his computation. This work was supported by the Science Committee of the Ministry of Science and Higher Education of the Republic of Kazakhstan (Grant No. BR 28713025). One of the \texttt{Python} scripts in the online addendum \cite{addendum} to the paper (\texttt{operad\_buchberger\_Xminus\_Xplus\_arity6.py})  was produced by ChatGPT-5.6 and then checked by the authors. 

\section{A short recollection of Gröbner bases for operads}

As indicated in the introduction, we shall use Gröbner bases for nonsymmetric operads, and we only briefly recall the relevant material, referring the reader to~\cite[Chapter 3]{MR3642294} for details.

Suppose that $\mathcal{P}=\{\mathcal{P}_n\}_{n\ge 1}$ is a nonsymmetric operad and $G$ is its Gröbner basis for a certain order of monomials. A Gröbner basis may be viewed as a collection of rewriting rules that allow one to eliminate all the leading terms. The defining property of a Gröbner basis is that this elimination is consistent, and $\mathcal{P}_n$ has a basis of cosets of monomials that are not divisible by the leading terms of elements of $G$, that is, do not contain those leading terms as fragments. 

Throughout this note, we compute Gröbner bases for a certain weight graded path-lexicographic order $\prec_{\mathrm{wpl}}$. Concretely, we define the weight function $\mathsf{wt}$ by 
 \[
\mathsf{wt}(y) = \mathsf{wt}(z) = 4, \quad \mathsf{wt}(x) = \mathsf{wt}(t) = 1,
 \]
and if for two monomials $\alpha$ and $\beta$, we have $\mathsf{wt}(\alpha)<\mathsf{wt}(\beta)$, we declare that $\alpha\prec_{\mathrm{wpl}}\beta$. If $\mathsf{wt}(\alpha)=\mathsf{wt}(\beta)$, we impose the path-lexicographic order $\prec_{\mathrm{pl}}$ for the order of operations
 \[
t \prec y \prec x \prec z.
 \]
The order $\prec_{\mathrm{pl}}$ works in the following way. To a monomial $\mu$ with $n$ arguments, we assign a sequence of words $(w_1(\mu),\ldots,w_n(\mu))$ in the alphabet $\{x,y,z,t\}$, where the word $w_k(\mu)$ is obtained by recording the sequence of operations leading to the  $k$-th argument of $\mu$: if the top level operation of $\mu$ is $v$, so that $\mu=v(\mu',\mu'')$ where $\mu'$ has $n'$ arguments, we set
 \[
w_k(\mu)=
\begin{cases}
\quad v w_k(\mu'), \qquad 1\le k\le n',\\
v w_{k-n'}(\mu''), \quad n'+1\le k\le n.
\end{cases}
 \]
To compare two monomials (with the same number of arguments) with respect to the  order $\prec_{\mathrm{pl}}$, we compare the corresponding sequences word by word, using the graded lexicographic order of words: we say that $\alpha\prec_{\mathrm{pl}}\beta$ if the smallest $k$ for which 
$w_k(\alpha)\ne w_k(\beta)$ satisfies $w_k(\alpha)\prec_{\mathrm{glex}} w_k(\beta)$.

\section{Hilbert series of the operad \texorpdfstring{$\mathcal{X^-}$}{Xminus}}\label{sec:Xminus}

In this section, we shall prove the following result.

\begin{thm} \label{th:Xminus}
For dimensions of components of the operad $\mathcal X^-$, we have 
    \[
    \dim_{\mathbb{Q}}\mathcal X^-_n=\binom{2n}{n-1}
    \]
for all $n\ne 3$.    
\end{thm}

\begin{proof}
As we indicated above, a computation using the program \cite{Dotsenko-Heijltjes} shows that the operad $\mathcal{X}^-$ has a finite Gröbner basis for the order $\prec_{\mathrm{wpl}}$; the input file for that computation is available in the online addendum \cite{addendum}. The rewriting system arising from that Gröbner basis consists of thirty-two rewriting rules, sixteen of arity three:
    \begin{gather}
        y(z(a_1, a_2), a_3) \mapsto y(x(a_1, a_2), a_3)  +  x(a_1, y(a_2, a_3)) - x(a_1, t(a_2, a_3)),\label{eq:ar3begin}\\
        y(y(a_1, a_2), a_3) \mapsto y(t(a_1, a_2), a_3)  -  x(a_1, y(a_2, a_3)) + x(a_1, t(a_2, a_3)), \\
        y(a_1, y(a_2, a_3))  \mapsto  y(x(a_1, a_2), a_3)  +  y(t(a_1, a_2), a_3) \label{eq:yry}, \\
        t(z(a_1, a_2), a_3)  \mapsto  - t(y(a_1, a_2), a_3)  +  y(a_1, t(a_2, a_3)), \\
        t(a_1, y(a_2, a_3))  \mapsto  t(a_1, t(a_2, a_3)), \\
        t(x(a_1, a_2), a_3)  \mapsto  - t(t(a_1, a_2), a_3)  +  t(a_1, t(a_2, a_3)), \\
        z(y(a_1, a_2), a_3)  \mapsto  z(t(a_1, a_2), a_3), \\
        y(a_1, z(a_2, a_3))  \mapsto  z(t(a_1, a_2), a_3), \\
        x(y(a_1, a_2), a_3)  \mapsto  y(a_1, x(a_2, a_3)), \\
        t(a_1, z(a_2, a_3))  \mapsto  t(a_1, x(a_2, a_3)), \\
        x(t(a_1, a_2), a_3)  \mapsto  t(a_1, x(a_2, a_3)), \\
        z(z(a_1, a_2), a_3)  \mapsto  z(x(a_1, a_2), a_3), \\
        z(a_1, z(a_2, a_3))  \mapsto  - z(a_1, y(a_2, a_3))  +  z(x(a_1, a_2), a_3), \\
        x(z(a_1, a_2), a_3)  \mapsto  z(a_1, x(a_2, a_3))  +  z(a_1, t(a_2, a_3)), \\
        x(a_1, z(a_2, a_3))  \mapsto  - x(a_1, y(a_2, a_3))  +  x(a_1, x(a_2, a_3))  +  x(a_1, t(a_2, a_3)), \\
        x(x(a_1, a_2), a_3)  \mapsto  x(a_1, x(a_2, a_3))  +  x(a_1, t(a_2, a_3))\label{eq:ar3end}
    \end{gather}
    and sixteen of arity four:
    \begin{gather}
        t(a_1, t(y(a_2, a_3), a_4))  \mapsto  t(a_1, t(t(a_2, a_3), a_4)), \label{eq:ar4begin}\\
        y(x(a_1, y(a_2, a_3)), a_4)  \mapsto  y(x(a_1, t(a_2, a_3)), a_4), \\
        y(a_1, x(a_2, y(a_3, a_4)))  \mapsto  y(a_1, x(a_2, t(a_3, a_4))), \\
        x(a_1, y(t(a_2, a_3), a_4))  \mapsto  x(a_1, t(t(a_2, a_3), a_4)), \\
        x(a_1, y(x(a_2, a_3), a_4))  \mapsto  - x(a_1, t(t(a_2, a_3), a_4))  +  x(a_1, t(a_2, t(a_3, a_4))), \\
        z(t(y(a_1, a_2), a_3), a_4)  \mapsto  z(t(t(a_1, a_2), a_3), a_4), \\
        z(x(a_1, y(a_2, a_3)), a_4)  \mapsto  z(x(a_1, t(a_2, a_3)), a_4), \\
        y(t(y(a_1, a_2), a_3), a_4)  \mapsto  y(t(t(a_1, a_2), a_3), a_4), \\
        x(a_1, x(a_2, y(a_3, a_4)))  \mapsto  x(a_1, x(a_2, t(a_3, a_4))), \\
        t(t(t(a_1, a_2), a_3), a_4)  \mapsto  - t(t(a_1, x(a_2, a_3)), a_4)  +  t(t(a_1, a_2), t(a_3, a_4)), \\
        t(t(y(a_1, a_2), a_3), a_4)  \mapsto  - t(y(a_1, x(a_2, a_3)), a_4)  +  t(y(a_1, a_2), t(a_3, a_4)), \\
        x(a_1, t(y(a_2, a_3), a_4))  \mapsto  x(a_1, t(t(a_2, a_3), a_4)), \\
        z(a_1, x(a_2, y(a_3, a_4)))  \mapsto  z(a_1, x(a_2, t(a_3, a_4))), \\
        x(a_1, y(a_2, t(a_3, a_4)))  \mapsto  x(a_1, t(a_2, t(a_3, a_4))), \label{eq:ar4end}\\
        x(a_1, y(a_2, x(a_3, a_4)))  \mapsto  x(a_1, t(a_2, x(a_3, a_4))),\label{eq:extra1}\\        
        t(a_1, x(a_2, y(a_3, a_4)))  \mapsto  t(a_1, x(a_2, t(a_3, a_4))).\label{eq:extra2}
    \end{gather}

An important feature of this Gröbner basis is that for each of its leading terms, one of the two arguments of the top level operation is the arity-one monomial $\mathrm{id}$. Thus, it is relatively easy to enumerate the normal forms, that is, the monomials that do not contain any of the leading terms as a fragment.
For that, we shall define the following sequences:
\begin{itemize}
\item[-] for $v \in \{ x, y, z, t\}$, we set $a_{v}(n)$ to be the number of normal forms $\alpha$ with $n$ arguments and top level operation $v$,
\item[-] for $u, v \in \{ x, y, z, t\}$, we set $a_{u,v}^l(n)$ to be the number of normal forms $\alpha$ with $n$ arguments and top level operation $v$ such that $u(\alpha,\mathrm{id})$ is a normal form (``$\alpha$ is left compatible with $u$''),
\item[-] for $u, v \in \{ x, y, z, t\}$, we set $a_{u,v}^r(n)$ to be the number of normal forms $\alpha$ with $n$ arguments and root $v$ such that $u(\mathrm{id},\alpha)$ is a normal form (``$\alpha$ is right compatible with $u$'').
\end{itemize}
We consider the following formal power series in one variable $q$:  
\begin{align*}
    A_{v}(q) &= \sum_{n \geq 1} a_{v}(n) q ^ n, \\
    A_{u,v}^l(q) &= \sum_{n \geq 1} a_{u,v}^l(n) q ^ n, \\
    A_{u,v}^r(q) &= \sum_{n \geq 1} a_{u,v}^r(n) q ^ n.
\end{align*}
Note that every normal form of $\mathcal X^-$ is either the unary  operation $\mathrm{id}$ or a monomial with a concrete top level operation, so the Hilbert series of $\mathcal X^-$ is given by the sum 
 \[
q + A_{x}(q) + A_{y}(q) + A_{z}(q) + A_{t}(q). 
 \]
 
The first group of equations relating these series to one another  corresponds to elements of arity three in the Gröbner basis. Namely, these series satisfy the following:
\begin{align}
A_{x}(q) &= q \cdot (q+A_{x, x}^r(q)+A_{x, y}^r(q)+A_{x, t}^r(q)), \label{eq:x}\\ 
A_{y}(q) &= (q+A_{y, x}^l(q)+A_{y, t}^l(q)) \cdot (q+A_{y, x}^r(q)+A_{t}(q)),\label{eq:y} \\ 
A_{z}(q) &= (q+A_{z, x}^l(q)+A_{z, t}^l(q)) \cdot (q+A_{z, x}^r(q)+A_{y}(q)+A_{t}(q)),\label{eq:z} \\ 
A_{t}(q) &= (q+A_{y}(q)+A_{t, t}^l(q)) \cdot (q+A_{t, x}^r(q)+A_{t, t}^r(q)),\label{eq:t}  
\end{align}

Indeed, if $v \in \{ x, y, z, t\}$, and $\alpha$ and $\beta$ are normal forms, then the monomial $v(\alpha,\beta)$ is a normal form if and only if $\alpha$ is left compatible with $v$ and $\beta$ is right compatible with $v$ in the sense of the definitions above. 
It remains to note that the displayed equations express exactly this enumeration scheme. Indeed, among the leading terms of the Gröbner basis we have all elements $x(v,\mathrm{id})$ with $v \in \{ x, y, z, t\}$, and $x(\mathrm{id},z)$, so the equation
 \[
A_{x}(q) = q \cdot (q+ A_{x, x}^r(q)+A_{x, y}^r(q)+A_{x, t}^r(q)) 
 \]
expresses the fact that a normal form with top level operation $x$ must have the arity-one monomial $\mathrm{id}$ as its first argument, and a normal form with top level operation other than $z$ or the arity-one monomial $\mathrm{id}$ as its second argument. The other equations are obtained in the exact same way.

The second group of equations relating these series to one another  corresponds to elements of arity four in the Gröbner basis.
Namely, these series satisfy the following:
\begin{align} 
A_{x,y}^r(q) &= q \cdot q, \\ 
A_{y,x}^l(q) &= q \cdot (q+A_{x, x}^r(q)+A_{x, t}^r(q)),\label{eq:yxl} \\ 
A_{y,x}^r(q) &= q \cdot (q+A_{x, x}^r(q)+A_{x, t}^r(q)), \label{eq:yxr}\\ 
A_{z,t}^l(q) &= (q+A_{t, t}^l(q)) \cdot (q+A_{t, x}^r(q)+A_{t, t}^r(q)),\label{eq:ztl} \\ 
A_{z,x}^l(q) &= q \cdot (q+A_{x, x}^r(q)+A_{x, t}^r(q)),\label{eq:zxl} \\ 
A_{y,t}^l(q) &= (q+A_{t, t}^l(q)) \cdot (q+A_{t, x}^r(q)+A_{t, t}^r(q)),\label{eq:ytl} \\ 
A_{x,x}^r(q) &= q \cdot (q+A_{x, x}^r(q)+A_{x, t}^r(q)),\label{eq:xxr} \\ 
A_{t,t}^l(q) &= q \cdot (q+A_{t, x}^r(q)+A_{t, t}^r(q)),\label{eq:ttl} \\ 
A_{t,t}^r(q) &= (q+A_{t, t}^l(q)) \cdot (q+A_{t, x}^r(q)+A_{t, t}^r(q)),\label{eq:ttr} \\ 
A_{t,x}^r(q) &= q \cdot (q+A_{x, x}^r(q)+A_{x, t}^r(q)),\label{eq:txr} \\ 
A_{z,x}^r(q) &= q \cdot (q+A_{x, x}^r(q)+A_{x, t}^r(q)),
\label{eq:zxr}\\ 
A_{x,t}^r(q) &= (q+A_{t, t}^l(q)) \cdot (q+A_{t, x}^r(q)+A_{t, t}^r(q)).\label{eq:xtr}
\end{align}

Indeed, suppose that $u, v \in \{ x, y, z, t\}$, and $\alpha$ and $\beta$ are normal forms for which $v(\alpha,\beta)$ is a normal form, and $u(v(\alpha,\beta),\mathrm{id})$ is a normal form.
If we denote by $\tau(\alpha)$ and $\tau(\beta)$ the top level operations of $\alpha$ and $\beta$ respectively, this means that $\alpha$ is left compatible with $v$ and $\beta$ is right compatible with $v$ in the sense of the definitions above, and also that the arity four operations $u(v(\tau(\alpha),\mathrm{id}),\mathrm{id})$ and $u(v(\mathrm{id},\tau(\beta)),\mathrm{id})$ are normal forms. 
It remains to note that the displayed equations express exactly this enumeration scheme. Indeed, among the leading terms of the Gröbner basis we have elements $t(x,\mathrm{id})$ and $t(z,\mathrm{id})$, as well as $z(t(y,\mathrm{id}),\mathrm{id})$, and $t(\mathrm{id},y)$ and $t(\mathrm{id},z)$, so the equation
 \[
A_{z,t}^l(q) = (q+A_{t, t}^l(q)) \cdot (q+A_{t, x}^r(q)+A_{t, t}^r(q)) 
 \]
expresses the fact that a normal form with top level operation $t$ which remains a normal form after inserting as the first argument of the operation $z$ must have a normal form with top level operation $t$ or the arity-one monomial $\mathrm{id}$ as its first argument, and a normal form with top level operation $x$ or $t$ or the arity-one monomial $\mathrm{id}$ as its second argument. The other equations are obtained in the exact same way.

To solve the above system of equations, we first note that the right hand sides of Equations \eqref{eq:yxl}, \eqref{eq:yxr}, \eqref{eq:zxl}, \eqref{eq:xxr}, \eqref{eq:txr}, and \eqref{eq:zxr} are the same, so their left hand sides are also the same, hence
\begin{equation}\label{eq:firstgp}
A_{y,x}^l(q)=A_{y,x}^r(q)=A_{z,x}^l(q)=A_{x,x}^r(q)=A_{t,x}^r(q)=A_{z,x}^r(q).     
\end{equation}
Similarly, the right hand sides of Equations \eqref{eq:ztl}, \eqref{eq:ytl}, \eqref{eq:ttr}, and \eqref{eq:xtr} are the same, so their left hand sides are also the same, hence
\begin{equation}\label{eq:secondgp}
A_{z,t}^l(q)=A_{y,t}^l(q)=A_{t,t}^r(q)=A_{x,t}^r(q).     
\end{equation}
Furthermore, these equalities allow us to rewrite Equation \eqref{eq:ttl} as
 \[
A_{t,t}^l(q) = q \cdot (q+A_{x, x}^r(q)+A_{x, t}^r(q)), 
 \]
implying that $A_{t,t}^l(q)$ is equal to all the elements \eqref{eq:firstgp}. 

Denoting the common value of all the elements \eqref{eq:firstgp} by $R$ and the common value of all the elements \eqref{eq:secondgp} by $S$, we see that Equations \eqref{eq:x}, \eqref{eq:y}, \eqref{eq:z}, \eqref{eq:t}, \eqref{eq:yxl} and \eqref{eq:ztl} become (if we drop the implied argument $q$ of all power series)
\begin{gather*}
A_x=q(q+q^2+R+S),\\
A_y=(q+R+S)(q+R+A_t),\\
A_z=(q+R+S)(q+R+A_y+A_t),\\
A_t=(q+R+S)(q+R+A_y),\\
R=q(q+R+S),\\
S=(q+R)(q+R+S).    
\end{gather*}
Let us denote $C:=q+R+S$. Then we have $R=qC$, $S=(q+R)C=qC(1+C)$, and therefore $C=q+R+S=q(1+C)^2$. Also, $A_y=C(q+R+A_t)$ and $A_t=C(q+R+A_y)$, which immediately implies 
\[
A_y=A_t=\frac{C(q+R)}{1-C}=\frac{qC(1+C)}{1-C}.
\]
Moreover, 
 \[
A_z=C(q+R+A_y+A_t)=qC(1+C)+2CA_y=qC(1+C)+\frac{2qC^2(1+C)}{1-C}
=\frac{qC(1+C)^2}{1-C}
 \]
and 
 \[
A_x=q(q+q^2+R+S)=q^3+qC. 
 \]
Recalling that the Hilbert series of the operad $\mathcal{X}^-$ is equal to $q + A_x + A_y + A_z + A_t$, we can now write it as
 \[
q+q^3+qC+\frac{2qC(1+C)}{1-C}+\frac{qC(1+C)^2}{1-C}=
q^3+\frac{q(1+C)^3}{1-C}.
 \]
Let us now recall that $C=q(1+C)^2$. Differentiating this with respect to $q$, we obtain
 \[
C'=(1+C)^2+2q(1+C)C'=(1+C)^2+\frac{2C}{1+C}C', 
 \]
implying
 \[
C'\frac{1-C}{1+C}=(1+C)^2, 
 \]
or 
 \[
C'=\frac{(1+C)^3}{1-C}. 
 \]
Therefore 
 \[
q^3+\frac{q(1+C)^3}{1-C}=q^3+qC'. 
 \]
Finally, $C=q(1+C)^2$ means that $C$ is the generating series for the Catalan numbers, 
 \[
C(q)=\sum_{n\ge 1}\frac{1}{n+1}\binom{2n}{n}q^n, 
 \]
so $qC'(q)$ is the generating function for 
 \[
\frac{n}{n+1}\binom{2n}{n}=\binom{2n}{n-1}, 
 \]
as required.
\end{proof}

\section{Dimensions of components of the operad \texorpdfstring{$\mathcal{X^+}$}{Xplus}}\label{sec:Xplus}

In this section, we shall prove the following result.

\begin{thm}\label{th:Xplus}
For dimensions of components of the operad $\mathcal X^+$, we have 
    \[
    \dim_{\mathbb{Q}}\mathcal X^+_n=\binom{2n}{n-1}+1
    \]
for all $n\ge 5$.    
\end{thm}

\begin{proof}
We shall once again use Gröbner bases for operads, though the argument is more subtle, since we were not able to find a finite Gröbner basis in this case. 

Let us denote by $G_i^+$ the arity $i$ part of the reduced Gröbner basis of the operad $\mathcal X^+$ for the order $\prec_{\mathrm{wpl}}$. Gröbner bases for operads are similar to Gröbner bases for \emph{homogeneous} associative noncommutative algebras: while usually infinite, they stabilize arity by arity, so the part of the reduced Gröbner basis in bounded arity is obtained by a finite computation, and as the first step of the proof, we  compute $G_i^+$ for all $i\le 6$.

To follow the proof, it will be particularly important to know the leading terms of all elements of the partial Gröbner basis we computed; the explicit formulas for all its elements are placed in the online addendum to this paper \cite{addendum}. 
 
The elements of $G_3^+$ have the same sixteen leading terms as the elements \eqref{eq:ar3begin}--\eqref{eq:ar3end}: these elements are obtained by echelonizing the matrix of coefficients of relations, and the single sign difference between the relations of $\mathcal X^-$ and the relations of $\mathcal X^+$ does not affect the positions of the pivots in the reduced row echelon form. However, $G_4^+$ has only fourteen elements, and their leading terms are the same as the leading terms of \eqref{eq:ar4begin}--\eqref{eq:ar4end}. It follows that the normal forms for $\mathcal{X}^+$ in arity four are the normal forms for $\mathcal{X}^-$ together with the leading terms 
\begin{equation}
x(a_1, y(a_2, x(a_3, a_4))),\quad
t(a_1, x(a_2, y(a_3, a_4)))\label{eq:extralead} 
\end{equation}
of \eqref{eq:extra1} and \eqref{eq:extra2}. (Of course, this accounts for the fact that $\dim\mathcal{X}_4^-=56$ but $\dim\mathcal{X}_4^+=58$.) 

Unlike what we observed for the operad $\mathcal{X}^-$, $G_5^+$ consists of eleven elements, and their leading terms are
    \[
\begin{gathered}
z(a_1,t(a_2,x(a_3,y(a_4,a_5)))), \qquad
t(y(a_1,a_2),x(a_3,y(a_4,a_5))), \\[2pt]
t(t(a_1,a_2),x(a_3,y(a_4,a_5))), \qquad
x(a_1,t(a_2,x(a_3,y(a_4,a_5)))), \\[2pt]
y(a_1,t(a_2,x(a_3,y(a_4,a_5)))), \qquad
z(t(a_1,x(a_2,y(a_3,a_4))),a_5), \\[2pt]
t(a_1,x(a_2,y(a_3,x(a_4,a_5)))), \qquad
t(t(a_1,x(a_2,y(a_3,a_4))),a_5), \\[2pt]
y(t(a_1,x(a_2,y(a_3,a_4))),a_5), \qquad
t(a_1,t(a_2,x(a_3,y(a_4,a_5)))), \\[2pt]
x(a_1, y(a_2, x(a_3, t(a_4, a_5)))).
\end{gathered}
    \]
Examining all of this data closely, we find the following. 

First of all, there are only two monomials of arity five that do not contain the leading terms of \eqref{eq:ar3begin}--\eqref{eq:ar4end} as fragments, but contain the first of \eqref{eq:extralead} as a fragment, namely
 \[
x(a_1, y(a_2, x(a_3, t(a_4, a_5)))),\qquad 
x(a_1, y(a_2, x(a_3, x(a_4, a_5)))),
 \]
where the first one is exactly the last leading term of $G_5^+$. Furthermore, the monomials of arity five that do not contain the leading terms \eqref{eq:ar3begin}--\eqref{eq:ar4end} as fragments but contain the second of \eqref{eq:extralead} as a fragment are 
\begin{gather*}
z(a_1,t(a_2,x(a_3,y(a_4,a_5)))), \qquad
t(y(a_1,a_2),x(a_3,y(a_4,a_5))), \\[2pt]
t(t(a_1,a_2),x(a_3,y(a_4,a_5))), \qquad
x(a_1,t(a_2,x(a_3,y(a_4,a_5)))), \\[2pt]
y(a_1,t(a_2,x(a_3,y(a_4,a_5)))), \qquad
z(t(a_1,x(a_2,y(a_3,a_4))),a_5), \\[2pt]
t(a_1,x(a_2,y(a_3,x(a_4,a_5)))), \qquad
t(t(a_1,x(a_2,y(a_3,a_4))),a_5), \\[2pt]
y(t(a_1,x(a_2,y(a_3,a_4))),a_5), \qquad
t(a_1,t(a_2,x(a_3,y(a_4,a_5)))), 
\end{gather*}
which are all but the last leading terms of $G_5^+$. It follows that the normal forms for $\mathcal{X}^+$ in arity five are the normal forms for $\mathcal{X}^-$ together with one additional element  
\begin{equation}\label{eq:extraar5}
x(a_1, y(a_2, x(a_3, x(a_4, a_5)))).     
\end{equation}
(Of course, this accounts for the fact that we have $\dim\mathcal{X}_5^+-\dim\mathcal{X}_5^-=211-210=1$.)

Continuing further, $G_6^+$ consists of exactly one element corresponding to a very simple rewriting rule
 \[
x(a_1, y(a_2, x(a_3, x(a_4, t(a_5, a_6))))) \mapsto x(a_1, t(a_2, x(a_3, x(a_4, t(a_5, a_6))))).
 \] 
Similarly to the above, there are only two monomials of arity six that do not contain the leading terms of $G_i^+$ for $i\le 5$ as fragments, but contain \eqref{eq:extraar5} as a fragment, namely
 \[
x(a_1, y(a_2, x(a_3, x(a_4, t(a_5, a_6))))),\qquad 
x(a_1, y(a_2, x(a_3, x(a_4, x(a_5, a_6))))),
 \]
and the first of them is the leading term of the only element of $G_6^+$. It follows that the normal forms for $\mathcal{X}^+$ in arity six are the normal forms for $\mathcal{X}^-$ together with one additional element  
\begin{equation}\label{eq:extraar6}
x(a_1, y(a_2, x(a_3, x(a_4, x(a_5, a_6)))).     
\end{equation}

At this point, the computation suggests that we can proceed by induction. That is, we shall prove that for each $n\ge 6$, $G_n^+$ contains exactly one rewriting rule $R_n$, which is given by 
\begin{align*}
x(a_1, y(a_2, x(a_3, \cdots, x(a_{n - 2}, t(a_{n - 1}, a_n)) \cdots))) \mapsto \\ x(a_1, t(a_2, x(a_3, \cdots, x(a_{n - 2}, t(a_{n - 1}, a_n)) \cdots))).
\end{align*}

The basis of induction is $n=6$ which our computation has already established. Suppose that in each arity $6\le i\le n$, the statement holds. The algorithm for computing the arity $n+1$ part of the Gröbner basis requires us to form all elements of arity $n+1$ that are overlaps of leading terms of elements of $G_i^+$ with elements of $G_j^+$ with $j\le n$, compute the S-polynomials for these overlaps, and then compute the reduced forms of those S-polynomials with respect to all $G_i^+$, with $i\le n$; the set $G_{n+1}^+$ is the result of echelonizing the resulting reduced forms; see \cite[Sec.~3.5.2]{MR3642294} for details.

The Gröbner basis we are exhibiting consists of two parts: the unstable part consisting of elements of $G_3^+$, $G_4^+$, and $G_5^+$, and the stable part consisting of elements of $G_k^+$ with $k\ge 6$. Therefore, when considering overlaps, there are three cases: overlaps of leading terms of two elements of the unstable part, overlaps of leading terms of two elements of the stable part, and overlaps of a leading term of an element of the unstable part with a leading term of an element of the stable part. 

We note that an overlap of two leading terms of elements of the unstable part has arity at most $5+5-2=8$, and one may check the claim for such arities by a direct computation using the program \cite{Dotsenko-Heijltjes}; the input file for that computation is available in the online addendum \cite{addendum}. 
Furthermore, for any $i\ge 6$, the leading term of $R_i$ does not form an overlap with the leading term of $R_j$ for any $j\ge 6$, so there are no overlaps of leading terms of two elements of the stable part.

Let us examine overlaps of a leading term of an element of the unstable part with a leading term of an element of the stable part. Every leading term of an element of the stable part is the right-normed product
 \[
x(a_1, y(a_2, x(a_3, \cdots, x(a_{k - 2}, t(a_{k - 1}, a_k))\cdots ))), \quad k\ge 6. 
 \]
We shall split the operations forming it into three groups: the \emph{top}  operations $x$, $y$, the \emph{bottom} operations $x$, $t$, and the remaining \emph{middle} operations $x$; note that since $k\ge 6$, we have at least one middle operation. It is important to note that most of the leading terms of the unstable part cannot form overlaps involving the middle operations: the only leading terms that can form such overlaps correspond to the rewriting rules
\begin{gather}
x(y(a_1, a_2), a_3)  \mapsto  y(a_1, x(a_2, a_3)),\label{eq:lc-xy}\\
x(t(a_1, a_2), a_3)  \mapsto  t(a_1, x(a_2, a_3)),\label{eq:lc-xt}\\
x(z(a_1, a_2), a_3)  \mapsto z(a_1, x(a_2, a_3))  +  z(a_1, t(a_2, a_3)),\label{eq:lc-xz}\\
x(x(a_1, a_2), a_3)  \mapsto x(a_1, x(a_2, a_3))  +  x(a_1, t(a_2, a_3)).\label{eq:lc-xx}
\end{gather}

The overlaps of these leading terms with the leading term of $R_k$ that involve middle operations can be examined similarly.

The reduction of the S-polynomial corresponding to an overlap of $R_n$ with \eqref{eq:lc-xy} depends on whether we consider the overlap over the middle $x$-operation closest to the top, or with some other one. In the former case, the S-polynomial corresponds to the rule
\begin{align*}
&x(a_1, y(a_2, y(a_3, x(a_4, \dots, x(a_{n - 1}, t(a_n, a_{n + 1})) \dots)))) \\
&\quad \mapsto x(a_1, t(a_2, x(y(a_3, a_4), x(a_5, \dots, x(a_{n - 1}, t(a_n, a_{n + 1})) \dots)))),
\end{align*}
whose leading term contains the leading term of (\ref{eq:yry}) as a fragment. The latter case is applicable only if $n \geq 7$, because that is when the middle part contains at least two operations. In that case, for the overlap at the position $i > 1$, the S-polynomial is equivalent to a rewriting rule

\begin{align*}
    x(a_1, y(a_2, x(a_3, \cdots, x(a_{i - 1}, y(a_i, x(a_{i + 1}, \cdots, x(a_{n - 1}, t(a_{n}, a_{n + 1})) \cdots))) \cdots))) \\ \mapsto x(a_1, t(a_2, x(a_3, \cdots, x(y(a_i, a_{i + 1}), x(a_{i + 2}, \cdots, x(a_{n - 1}, t(a_{n}, a_{n + 1})) \cdots)) \cdots))),
\end{align*}
which contains the leading term of $R_{n - (i - 1)}$ along the subsequence of operations in the interval $(i - 1)$ and $n$. 

The overlaps between the middle part of $R_n$ and (\ref{eq:lc-xt}) - (\ref{eq:lc-xx}) are treated similarly. The reduction to zero of the S-polynomial corresponding to an overlap of $R_n$ with \eqref{eq:lc-xt} is immediate using the element
 \[
x(a_1, y(a_2, t(a_3, a_4))) \to  x(a_1, t (a_2, t (a_3, a_4))) 
 \]
for the overlap over the middle $x$-operation closest to the top  
and using $R_i$ for some $i<n$ otherwise. The reduction of the S-polynomial corresponding to an overlap of $R_n$ with \eqref{eq:lc-xz} depends on whether we consider the overlap over the middle $x$-operation closest to the top, or with some other one. In the former case, the reduction to zero is done using several elements of $G_3^+$ and $G_4^+$, while in the latter case, the S-polynomial is reduced to zero using various $R_i$ for $i<n$. Finally, the S-polynomial corresponding to an overlap of $R_n$ with \eqref{eq:lc-xx} reduces to $R_{n+1}$ using $R_i$ for various $i\le n$.

In all other overlaps of a leading term of an element of the unstable part with a leading term of an element of the stable part, the common operations of the two leading terms are either the top level or the bottom level operations of the stable leading term. We note that the corresponding S-polynomial and its reductions are computed without touching the middle operations. This allows us to use the following computational trick. We adjoin to our operad a binary operation $\lambda$, and one additional relation 
 \[
x(a_1, y(a_2, \lambda(a_3, x(a_4, t(a_5, a_6))))) \mapsto x(a_1, t(a_2, \lambda(a_3, x(a_4, t(a_5, a_6))))).
 \]
(This operation $\lambda$ is nothing but the compressed form of the iterated middle operations.) If one computes the Gröbner basis for this operad up to arity $6+5-2=9$ using the program \cite{Dotsenko-Heijltjes} (the input file for that computation is available in the online addendum \cite{addendum}), and replaces $\lambda$ by a right-normed product of several operations $x$, the result consists of the reduced forms of the S-polynomials being considered. A simple \texttt{Python} script (available in the online addendum \cite{addendum}) may be used to check that this only produces elements $R_s$ for various $s$. This proves our assertion about the Gröbner basis. 

To conclude the argument, we note that we in particular proved that in every arity $n\ge 5$, monomials that do not contain leading terms of elements of $G_3^+$, $G_4^+$, and $G_5^+$ as fragments do not contain the second of \eqref{eq:extralead} as a fragment. Thus, for $n\ge 5$, normal forms for $\mathcal{X}_n^+$ that are not normal forms for $\mathcal{X}_n^-$ may only contain the first of \eqref{eq:extralead}. Moreover, since the left argument of an operation labelled $x$ in a normal form must be the arity-one monomial $\mathrm{id}$, one immediately sees by induction on $n\ge 6$ that there are only two monomials that are not normal forms for $\mathcal{X}^-$ and do not contain the leading terms of $R_i$, $6\le i< n$, that is,  
\begin{gather*}
x(a_1, y(a_2, x(a_3, \cdots, x(a_{n - 2}, t(a_{n - 1}, a_n))\cdots ))),\\    
x(a_1, y(a_2, x(a_3, \cdots, x(a_{n - 2}, x(a_{n - 1}, a_n))\cdots ))),
\end{gather*}
and the first of them is the leading term of $R_n$. Hence, we have only one extra (in comparison with normal forms for $\mathcal{X}^-$) normal form in each arity, which proves the required dimension formula.
\end{proof}

\section{On Koszulness of operads \texorpdfstring{$\mathcal{X}^{\pm}$}{Xpm}}\label{sec:Koszul}

In \cite[Sec.~4.2]{Loday}, it is recalled that if an operad $\mathcal{X}$ ($=\mathcal{X}^+$ or $\mathcal{X}^-$) is self-dual for the Koszul duality, then for each $k\ge 1$, there is a finite chain complex $C^{(k)}_\bullet(\mathcal{X})$ with 
 \[
C^{(k)}_r(\mathcal{X})\cong\bigoplus_{m_1+\cdots+m_r=k}\mathcal{X}_r\otimes\mathcal{X}_{m_1}\otimes\cdots\otimes\mathcal{X}_{m_r}, 
 \]
and acyclicity of these complexes for all $k>1$ is equivalent to the Koszul property of~$\mathcal{X}$; this is essentially a way to present the Koszul complex of the operad $\mathcal{X}$ \cite{MR1301191,LodayVallette} to non-experts. It is then asserted that this shows that if $\mathcal{X}$ is Koszul, then $\dim\mathcal{X}_n=4^{n-1}$. (And it is precisely the property $\dim\mathcal{X}_n=4^{n-1}$ which was then disproved in \cite[Prop.~2.4.3.3]{Laubie} for both operads $\mathcal{X}^+$ and $\mathcal{X}^-$.) However, this assertion seems to be unfounded. Indeed, the acyclicity of the complexes $C^{(k)}_\bullet(\mathcal{X})$ for $k>1$ implies that the Hilbert series $f_{\mathcal{X}}(q)$
satisfies 
 \[
f_{\mathcal{X}}(-f_{\mathcal{X}}(-q))=q, 
 \]
but that property together with $\dim\mathcal{X}_2=4$, $ \dim\mathcal{X}_3=16$ is not enough to show that $\dim\mathcal{X}_n=4^{n-1}$ for all $n$. Indeed, if 
 \[
f_{\mathcal{X}}(q)=q+4q^2+16q^3+aq^4+bq^5+O(q^6), 
 \]
we have 
 \[
f_{\mathcal{X}}(-f_{\mathcal{X}}(-q))=
 q + (-24a + 2b + 1024)q^5 + O(q^6)
 \]
which does not allow us to determine $a$ and $b$ individually: we can only infer 
 \[
12a=b+512. 
 \]
However, for the operad $\mathcal{X}^-$, we have $a=56$ and $b=210$, and for the operad $\mathcal{X}^+$, we have $a=58$ and $b=211$, and neither pair satisfies the necessary constraint. Thus, neither of the operads $\mathcal{X}^+$ and $\mathcal{X}^-$ is Koszul.

\section{Case of a ground field of positive characteristic}\label{sec:char-p}

In this section, we describe in detail what happens if we consider our operads over the ground field $\mathbb{F}_p$. It is \emph{a priori} clear that some dependence of characteristic is there: for instance, for $p=2$, we have $\mathcal{X}^-=\mathcal{X}^+$, so it is clear that the dimension formulas proved above cannot be true at the same time. 

\begin{thm}\leavevmode
\begin{itemize}
\item For $p\ne 2$, we have 
\begin{gather*}
\dim_{\mathbb{F}_p}\mathcal{X}^-_n=\binom{2n}{n-1}\quad  \text{ for }n\ne 3,\\
\dim_{\mathbb{F}_p}\mathcal{X}^+_n=\binom{2n}{n-1}+1\quad  \text{ for } n\ge 5.
\end{gather*}
\item For $p=2$, we have
 \[
\dim_{\mathbb{F}_p}\mathcal{X}^-_n= \dim_{\mathbb{F}_p}\mathcal{X}^+_n=\binom{2n}{n-1}+1\quad  \text{ for } n\ge 5.
 \]
\end{itemize}
\end{thm}

\begin{proof}
Let us begin with a general useful remark. Suppose that an operad $\mathcal{P}$ is defined over $\mathbb{Z}$, and that $G$ is a system of relations of that operad which is, for a given order of monomials, a Gröbner basis of the operad $\mathbb{Q}\otimes_\mathbb{Z}\mathcal{P}$. Then, if the leading coefficients of the elements of $G$ are coprime to $p$, then $G$ is a Gröbner basis of $\mathcal{P}$ over $\mathbb{F}_p\otimes_\mathbb{Z}\mathcal{P}$. This follows from the Diamond Lemma \cite[Sec. 3.5.1]{MR3642294}: testing the Gröbner basis property amounts to check that all S-polynomials reduce to zero, and for that we only ever need to divide by the leading coefficients, so this computation survives a reduction modulo~$p$.

First, a direct computation (implemented in a \texttt{Python} script available in the online addendum \cite{addendum}) shows that the computation of the arity four part of the reduced Gröbner basis for $\mathcal{X}^-$ for our order of monomials only creates some leading coefficients equal to $2$. Thus, for any characteristic $p\ne 2$ the operad $\mathcal{X}^-$ has the same reduced Gröbner basis as the one presented in the proof of Theorem \ref{th:Xminus}. 

Next, a direct computation (implemented in a \texttt{Python} script available in the online addendum \cite{addendum}) shows that the computation of the arity at most six part of the reduced Gröbner basis for $\mathcal{X}^+$ for our order of monomials does not involve any division by a nontrivial scalar: the leading coefficients are automatically $\pm1$. Moreover, in the proof of Theorem \ref{th:Xplus}, we established that for $n\ge 6$, each element $R_{n+1}$ can be obtained by reducing the S-polynomial of $R_n$ with the arity three relation \eqref{eq:lc-xx}; by direct inspection, in the process of this reduction, $R_{n+1}$ appears with coefficient $1$. Thus, for any characteristic $p>0$, the operad $\mathcal{X}^+$ has the same reduced Gröbner basis as the one presented in the proof of Theorem \ref{th:Xplus}.
  
Finally, over a field of characteristic $2$, the operads $\mathcal{X}^-$ and $\mathcal{X}^+$ coincide, so the previous argument handles this case as well, completing the proof. 
\end{proof}

\printbibliography

\end{document}